\documentclass{amsart}
\usepackage{amssymb, latexsym, euscript}
\usepackage{graphicx}

\DeclareMathOperator{\ins}{int}

\usepackage{color}

      \def\dC{{\mathbb C}}

      \def\dR{{\mathbb R}}

\def\cD{{\mathcal D}}

\def\bm\chi{\mbox{\boldmath$\chi$}}

\def\RE{{\rm Re\,}}
\def\IM{{\rm Im\,}}

\def\eps{{\varepsilon}}

\def\e{{\rm e}}

\def\supp{{\rm supp\,}}

\def\cmr{{\dC \setminus \dR}}
\def\set#1{\left\{#1\right\}}
\def\sbs{\subseteq}
\newcommand{\ii}{\textrm{i}}

\newcommand{\al}{\alpha}
\newcommand{\be}{\beta}

\renewcommand{\Re}{\hbox{\rm Re }}
\renewcommand{\Im}{\hbox{\rm Im }}

\newtheorem{theorem}{Theorem}[section]
\newtheorem{proposition}[theorem]{Proposition}

\theoremstyle{definition}
\newtheorem{example}[theorem]{Example}

\numberwithin{equation}{section}

\begin{document}

\title[Global and local behavior of zeros  of nonpositive type]
{Global and local behavior of zeros  of nonpositive type}

\author[H.S.V.~de~Snoo]{Henk de Snoo}
\address{Johann Bernoulli Institute for  Mathematics and Computer Science\\
University of Groningen \\
P.O. Box 407, 9700 AK Groningen \\
Nederland} \email{desnoo@math.rug.nl}

\author[H.~Winkler]{Henrik Winkler}
\address{
Institut f\"{u}r Mathematik\\
Technische Universit\"at Ilmenau \\
Curiebau, Weimarer Str.~25, 98693 Ilmenau \\
Germany} \email{henrik.winkler@tu-ilmenau.de}

\author[M.~Wojtylak]{Micha\l{} Wojtylak}
 \address{
Institute of Mathematics\\
Jagiellonian University\\
 \L ojasiewicza 6\\
 30-348 Krak\'ow\\
Poland} \email{michal.wojtylak@gmail.com}

\thanks{The work of Micha\l{} Wojtylak was supported by the Alexander von Humboldt Foundation}

\translator{}


\begin{abstract}
A generalized Nevanlinna function $Q(z)$ with one negative square has precisely
one generalized zero of nonpositive type in the closed extended upper halfplane.
The fractional linear transformation defined by $Q_\tau(z)=(Q(z)-\tau)/(1+\tau Q(z))$,
$\tau \in \dR \cup \{\infty\}$, is a generalized Nevanlinna function with one negative square.
Its generalized zero of nonpositive type $\alpha(\tau)$ as a function of $\tau$ is being studied. In particular, it is shown that it is continuous and its behavior in the points where the function extends through the real line is investigated.
\end{abstract}

\date{ \today; Filename: \jobname. }

\maketitle

\section{Introduction}

Let $M(z)$ be an ordinary Nevanlinna function, i.e.,
a function which is holomorphic in $\dC^{+}$ and
which maps the upper half--plane into itself.
It is well known that $M(z)$  admits a representation
\begin{equation}\label{neva}
M(z)=a+bz+\int_{\dR} \left(\frac 1{t-z}-\frac{t}{t^2+1}\right) d\sigma(t),
\quad z \in \cmr,
\end{equation}
with $a\in\dR$, $b>0$, and a measure $\sigma$ satisfying
$\int_{\dR} d\sigma(t)/(t^{2}+1)<\infty$; cf. \cite{donoghue}.
In the lower half--plane $\dC^{-}$ the function $M(z)$ is defined by the symmetry
principle $\overline{M(z)}=M(\bar z)$.
Then $M(z)$ is holomorphic on $\dC^+\cup\dC^-\cup(\dR\setminus\supp\sigma)$.
Note that if $\dR\setminus\supp\sigma$ contains some interval $I$, then the extension of $M(z)$ given
on $\dC^{+}$ to the set $\dC^+\cup\dC^-\cup I$ is given by the Schwarz reflection principle.
However, the main interest in this paper is in the situation when $M(z)$ extends
holomorphically across the real line, but the extension does \textit{not necessarily}
satisfy the symmetry principle $\overline{M(z)}=M(\bar z)$. The simplest example
is given by
 $$
 M(z)= \ii =\frac{1}{\pi}\int_{\dR} \left( \frac 1{t-z}-\frac{t}{t^2+1} \right) dt, \quad z \in \cmr.
 $$
 It appears that such an extension is possible if and only if the measure
 $\sigma$ in \eqref{neva} is absolutely continuous and its derivative extends to a
 holomorphic function; see Theorem \ref{hExt} for details.

Let $Q(z)$ be a generalized Nevanlinna function of class
$\mathbf{N}_1$, i.e., a meromorphic function in the upper
half--plane such that the kernel
 $$
 \mathsf{N}_{Q}(z,w)=\frac{Q(z)-\overline{Q(w)}}{z-\bar w}, \quad z,w \in \dC^{+},
 $$
has precisely one negative square. It has been shown
 that $Q(z)$ has a unique factorization
\begin{equation}\label{nev}
Q(z)=R(z)\,M(z),
\end{equation}
with $R(z)$ of one of the following three forms
\begin{equation}\label{nev1}
\frac{(z-\alpha)(z-\bar \alpha)}{(z-\beta)(z-\bar\beta)},
\quad (z-\alpha)(z-\bar \alpha), \quad \frac1{(z-\beta)(z-\bar\beta)},
\end{equation}
with $\al,\be\in\dC^+\cup\dR$; cf. \cite{DHS1,DLLSh}. The point
$\alpha$ is called the \textit{generalized zero of nonpositive type}
(GZNT) of $Q(z)$
and the point $\beta$ is called the \textit{ generalized pole of nonpositive type}
(GPNT) of $Q(z)$ ;
see e.g. \cite{DHS1,DLLSh,L} for a characterization of GZNT and GPNT
in terms of nontangential limits.

Each $\mathbf{N}_1$  function $Q(z)$ has an operator representation
with a selfadjoint operator $H$ in a Pontryagin space with one negative square, see \eqref{QHrel}. One of the reasons to consider
not necessarily symmetric extendable Nevanlinna functions $M(z)$   is
that then the function $Q(z)$ defined by \eqref{nev}   lead to operator models for  each of the cases of the
following classification of eigenvalues of nonpositive type of $H$
(see \cite{DHS3,JL85,KL73,KL77,langerspec}):
\begin{itemize}
\item[(A)] $\al$ is an eigenvalue of negative type of $H$ with algebraic multiplicity 1;
\item[(B)] $\al$ is a singular critical eigenvalue of $H$ with algebraic multiplicity 1;
\item[(C)] $\al$ is a regular critical eigenvalue of $H$ with algebraic multiplicity 2;
\item[(D)] $\al$ is a singular critical eigenvalue of $H$ with algebraic multiplicity 2;
\item[(E)] $\al$ is a regular critical eigenvalue of $H$ with algebraic multiplicity 3,
\end{itemize}
as will be shown in Section \ref{5}.

A function $Q(z)$ in $\mathbf{N}_1$ generates of family of functions $Q_\tau(z)$ via the  linear fractional transformation
$$
Q_\tau(z):=\frac{Q(z)-\tau}{1+\tau Q(z)},\quad \tau\in\dR,
$$
and by
$$
Q_\infty(z):=-\frac{1}{Q(z)}, \quad \tau=\infty.
$$
It is known that $Q_\tau(z)\in\mathbf{N}_1$, which allows to define for $\tau\in\dR\cup\{\infty\}$ the numbers $\alpha(\tau)$ and $\beta(\tau)$ as, respectively,
GZNT and GPNT of the function $Q_{\tau}(z)$.
The local properties of  $\al(\tau)$ in the case when $\al(\tau_0)$
lies in a spectral gap of $M(z)$ were investigated in detail in \cite{SWW}.
In the present work these results are generalized to the case when $Q(z)$ extends holomorphically to the lower half--plane around
$\al(\tau_0)$.
The paper also contains some results of a global nature
concerning the function $\tau \to \al(\tau)$. In particular, in Theorem \ref{kurve} it is
shown that $\al(\tau)$ forms a curve on the Riemann sphere which is homeomorphic to a circle.
This problem was still open in \cite{SWW} and is now solved by means of recent
results concerning the
convergence behavior of generalized Nevanlinna functions \cite{LaLuMa}. The problem is related to the convergence of poles in Pad\'e
approximation, see \cite{DeDe,Ra}.
The authors are indebted to 
Maxim Derevyagin for assistance in this matter. A related problem of tracking the eigenvalue of nonpositive type in the context of random matrices was considered in \cite{PaWo,Wojtylak12b}.

\section{Non-symmetric extensions of Nevanlinna functions}

A \textit{Nevanlinna function} $M(z)$ is a function
which is defined and holomorphic on $\cmr$,
which is symmetric with respect to $\dR$, i.e.,
$\overline{M(z)}=M(\bar{z})$, $z \in \cmr$,
and which maps $\dC^{+}$ into $\dC^{+} \cup \dR$.
If for some $z_{0} \in \cmr$ the value $M(z_{0})$ is real,
then the function $M(z)$ is constant on $\dC$.
In the following mostly
nonconstant Nevanlinna functions
appear.

Let $I \subset \dR$ be an open interval such that
$\IM M(z_{n})\to 0$ for any sequence $(z_{n})$  from $\dC^{+}$
or $\dC^{-}$ converging to a point in $I$.  According
to the Schwarz reflection principle, the function
$M(z)$ has a holomorphic continuation across $I$ and
$M(x) \in \dR$ for $x \in I$, or, equivalently, the support
of $\sigma$ in the integral representation  \eqref{neva}
of $M(z)$ has a gap on $I$.

However, it may happen that a Nevanlinna function $M(z)$,
restricted to $\dC^{+}$, has a holomorphic continuation
$\widetilde M(z)$ across an interval $I \subset \dR$
without $\widetilde M(z)$ being real for $z \in I$.
In this case the symmetry property will be lost.
For example the constant function $M(z)=\ii$,
as defined on $\dC^{+}$, extends to a holomorphic
function on all of $\dC$  (although $\supp\sigma=\dR$).

In the sequel Nevanlinna functions are considered on $\dC^{+}$ and their
holomorphic continuations across (a part of) $\dR$ are being studied.

\begin{theorem}\label{hExt}
Let  $M(z)$ be an Nevanlinna function of the form
\eqref{neva}  and let $\Omega$
be a simply connected domain, symmetric with respect to $\dR$.
Then the following statements are equivalent:
\begin{enumerate}\def\labelenumi{\rm (\roman{enumi})}
\item the restriction of $M(z)$ to $\dC^{+}$ extends
to a holomorphic function in $\Omega\cup\dC^+$;
\item the measure $d\sigma$  in \eqref{neva} satisfies
$$
d \sigma(t) =\phi(t)\,dt, \quad t\in\Omega\cap\dR,
$$
where $\phi(z)$ is a real holomorphic function on $\Omega$.
\end{enumerate}
\end{theorem}

\begin{proof}
Note that the conditions imply that $\Omega\cap\dR\neq\emptyset$.

(i) $\Rightarrow$ (ii)
Denote the holomorphic extension of $M(z)$ onto $\dC^+\cup\Omega$
by $\widetilde M(z)$. Note that the function
$\phi(z)$ on $\Omega$, given by
$$
 \phi(z) = \frac{ \widetilde M(z) - \overline{\widetilde M(\bar z)}}{2\pi \ii},
 \quad z \in\Omega,
$$
is well-defined and holomorphic on  $\Omega$.  Furthermore, one has
\[
 \phi(t) = \frac{ \widetilde M(t) - \overline{\widetilde M(t)}}{2\pi \ii}
 =\frac{1}{\pi} \,\IM \widetilde M(t),
 \quad t \in \Omega \cap \dR.
\]
Let $t_{1}< t_{2}$ belong to the open connected set $\Omega \cap \dR$.
Then the Stieltjes inversion formula
$$
\sigma(t_2)-\sigma(t_1)
=\lim_{\eps\downarrow 0}\frac1\pi\int_{t_{1}}^{t_{2}} \Im M(t+\ii\eps)\,dt,
$$
can be rewritten as
\[
\begin{split}
\sigma(t_2)-\sigma(t_1)
&=\lim_{\eps\downarrow 0}\frac1\pi\int_{t_1}^{t_2} \Im \widetilde M(t+\ii\eps)\,dt \\
& =\frac1\pi\int_{t_1}^{t_2} \Im \widetilde M(t)\,dt
=\int_{t_1}^{t_2} \phi(t)\,dt,
\end{split}
\]
since for $\eps>0$ sufficiently small one has the inclusion
\[
\{z=t+\ii y \in \dC:\, t \in [t_1,t_2], \, 0 \le y \le \eps\} \sbs\Omega.
\]
Since $t_1 < t_2$ are arbitrary,
it follows that the measure $d \sigma$ is absolutely continuous
and that $d\sigma(t)=\phi(t)\,dt$ on $\Omega\cap\dR$.

(ii) $\Rightarrow$ (i) Assume that $d\sigma(t)=\phi(t)\,dt$, $t\in\Omega\cap\dR$,
with  $\phi(z)$ holomorphic on $\Omega$.
Let $t_{1}< t_{2}$ be real numbers such that $[t_1,t_2]\sbs \Omega\cap\dR$
and let $\Omega^-=\Omega\cap\dC^-$.
It will be shown that $M(z)$ extends holomorphically to
$$
\Omega_0=\dC^+\cup(t_1,t_2)\cup\Omega^-.
$$
This will imply (i), since $t_{1},t_{2}$ are arbitrary and
$\Omega$ is simply connected.

In order to prove the claim, let $\Gamma_0$ be a curve that joins
$t_2$ with $t_1$, lying entirely in $\Omega\cap\dC^+$,
except for its boundary points. Let $\Gamma=\Gamma_0\cup[t_1,t_2]$
with the usual orientation.
The interior of $\Gamma$ is denoted by $\ins\Gamma$.
Write the function $M(z)$ as
$$
M(z)=\int_{t_1}^{t_2}\frac{ \phi(t)}{t-z}\,dt+\tilde a +bz+
\int_{\dR\setminus[t_1,t_2]}\left(\frac1{t-z}-\frac{t}{t^2+1}\right)d\sigma(t),
$$
with the constant $\tilde a$ given by
\[
\tilde a=a-\int_{t_1}^{t_2} \frac{t\ \phi(t)}{t^2+1}\,dt.
\]
Note that the function
$$
\tilde a +bz+
\int_{\dR\setminus[t_1,t_2]} \left( \frac1{t-z}-\frac{t}{t^2+t} \right) d\sigma(t)
$$
can be extended to a holomorphic function in
$\dC^+\cup(t_1,t_2)\cup\dC^-$ via the
Schwarz reflection principle. Hence, to show that
$M(z)$ extends holomorphically to $\Omega_0$,
it suffices to show that the function
$$
M_1(z)=\int_{t_1}^{t_2}\frac{ \phi(t)}{t-z}\,dt
$$
can be extended to a holomorphic function across $(t_{1}, t_{2})$. 
According to Cauchy's theorem one has
$$
\phi(z)=\frac1{2\pi\ii}\int_{\Gamma}\frac{\phi(\zeta)}{\zeta-z}\,d\zeta,
\quad z\in\ins\Gamma,
$$
and, hence,
\begin{equation}\label{berl}
M_1(z)
=2\pi\ii\phi(z)+\int_{\Gamma_0}\frac{\phi(\zeta)}{\zeta-z}\,d\zeta,
\quad z\in\ins\Gamma.
\end{equation}
The function $\phi(z)$ is holomorphic in $\Omega_0$ and the formula
\[
\int_{\Gamma_0}\frac{\phi(\zeta)}{\zeta-z}\,d\zeta,
\quad z\in \dC\setminus\Gamma_0,
\]
defines a holomorphic function in the simply connected domain
$$
A=\Omega^-\cup(t_1,t_2)\cup\ins\Gamma.
$$
Therefore, the right-hand side of \eqref{berl}
defines a holomorphic function in
$A$
while $M_1(z)$ is a holomorphic function in
$\dC^+$. Note that $A\cap \dC^+=\ins\Gamma$,
which is a connected set, and $A\cup \dC^+=\Omega_0$.
Hence,  $M_1(z)$ extends holomorpically  to $\Omega_0$.
\end{proof}

\begin{theorem}\label{QExt}
Let $Q(z)$ be in $\mathbf{N}_{1}$
with the representation $Q(z)=R(z)M(z)$
as in \eqref{nev} and \eqref{nev1} with
$\alpha, \beta \in \dC^{+} \cup \dR$.
Let $\Omega$ be a simply connected domain, symmetric with respect to $\dR$
and assume that $\beta \notin\Omega$ and $\alpha\in\Omega\cap\dR$.
Then the following statements are equivalent:
\begin{enumerate}\def\labelenumi{\rm (\roman{enumi})}
\item
the restriction of $Q(z)$ to $\dC^{+} \setminus \{\beta\}$
extends to a holomorphic function in $\Omega\cup\dC^+ \setminus \{\beta\}$;
\item the function $M(z)$ is of the form
\begin{equation}\label{dM}
M(z) =
                M_1(z) +\frac{m_0}{\alpha-z}, 
\quad t\in\Omega\cap\dR,
\end{equation}
where $M_1(z)$ is a Nevanlinna function such that the restriction of $M_1(z)$ to $\dC^+$ extends to a holomorphic function in $\dC^+\cup\Omega$ and $m_0\ge0$;
\item
the measure $d \sigma$ for $M(z)$ in \eqref{neva} satisfies
\begin{equation}\label{dsigma}
d\sigma(t) =
                \phi(t) dt +m_0 \delta_{\alpha}(t),
\quad t\in\Omega\cap\dR,
\end{equation}
where $\phi(z)$ is a real holomorphic function on $\Omega$, $m_{0} \ge 0$, and
$\delta_{\alpha}$ is the Dirac measure at $\alpha$.
\end{enumerate}
If instead of  $\alpha\in\Omega\cap\dR$ one assumes $\alpha \notin\Omega$, then the equivalences above hold with $m_0=0$ in statements {\rm(ii)} and {\rm(iii)}.

\end{theorem}

\begin{proof}
The proof will be given for the case $\alpha\in\Omega\cap\dR$;
the proof in the case $\alpha\notin\Omega$ is  left as an exercise.

(i) $\Rightarrow$ (ii)
Let $\widetilde Q(z)$ be a holomorphic extension of $Q(z)$ to  $\Omega\cup\dC^+\setminus\set\beta$ and let
$$
 \widetilde M(z)=R_0^{-1}(z)\frac1{(z-\al)^2}\widetilde Q(z),
$$
 where $R_{0}(z)=1/(z-\beta)(z-\bar\beta)$
or $R_{0}(z)\equiv 1$, depending on the position of the GPNT $\beta$. Note that $\widetilde M(z)$
is holomorphic in $\Omega\cup\dC^+\setminus\set{\al}$ and that $\widetilde M(z)=M(z)$ for $z\in\dC^+$.
%
Observe that
$$
\widetilde Q(\al)=\lim_{z\widehat\to \alpha} Q(z)=0,
$$
and, hence,  $\alpha$ is a pole of order at most one of $\widetilde M(z)$.
If
$$
\widetilde Q'(\alpha)=0,
$$
then $\widetilde M(z)$ is holomorphic at $\al$, and by Theorem \ref{hExt}
one gets a representation \eqref{dsigma} with $m_0=0$.
If $\widetilde Q'(\al)\neq 0$, then it follows from
$$
\lim_{z\widehat\to \al} (z-\al)\ M(z)
=R_0^{-1}(\al)\,{\widetilde Q'(\al) },
$$
that the measure $d\sigma$ in the representation \eqref{neva} of $M(z)$ has a mass point at $\al$. Therefore the function
\begin{equation}\label{M1}
 M_1(z)= M(z)-\frac{m_0}{\al-z} \quad \mbox{with} \quad m_0= \lim_{z\widehat\to \al} (z-\al)\ M(z),
\end{equation}
is a Nevanlinna function. Put
$$
\widetilde M_1(z)=\widetilde M(z) -\frac{m_0}{\al-z},
$$
and note that
$$
\lim_{z\widehat\to \al} (z-\al)\ \widetilde M_1(z)
=0.
$$
Thus,
$\widetilde M_1(z)$ is holomorphic at $\al$ and in consequence in $\dC^+\cup\Omega$.

The implication (ii) $\Rightarrow$ (i) is obvious and the equivalence (ii) $\Leftrightarrow$ (iii) is a direct consequence of
of  Theorem \ref{hExt}.
\end{proof}

\section{Global properties of the function $\al(\tau)$}\label{2}

Before continuing with extension properties across $\dR$ an open problem
from \cite{SWW} will be considered. To state
it, a notion of convergence for a class of meromorphic
functions is required; see \cite{LaLuMa} for the original treatment.
Let $\cD$ be a nonempty open subset of the complex plane,
and let $(Q_n)$ be a sequence of functions which are meromorphic on $\cD$.
The sequence $(Q_n)$ is said to \textit{converge locally uniformly on} $\cD$
to the function $Q$, if for each nonempty open set $\cD_0\sbs\dC$
with compact closure $\overline\cD_0\sbs\cD$
there exists an index $n_0(\cD_0)$ such that for $n>n_0(\cD_0)$
the functions $(Q_n)$ are holomorphic on $\cD_0$ and
 $$
 \lim_{n\to\infty} Q_n(z)=Q(z),\quad \text{uniformly on }\cD_0.
 $$
The extended complex plane is denoted by $\overline\dC$.

\begin{theorem}\label{cont}
Let $Q(z)$ be an $\mathbf{N}_1$ function and
let $\tau_n \in \dR$ converge to $\tau \in \overline\dR$.
Then $Q_{\tau_n}(z)$
converges locally uniformly to
$Q_{\tau}(z)$ on $\dC^+\setminus\set{\beta(\tau)}$.
\end{theorem}

\begin{proof}
Let $\cD$  be some open, bounded subset of $\dC^+$ with $\overline\cD\subset\dC^+\setminus\{\beta(0),\beta(\tau)\}$ and let $\tau_n\to\tau$.
Consider first the case $\tau\in\dR$. Since $1+\tau Q(z)$ has no zero on $\overline\cD$, it follows that
$$
\inf_{z\in\cD} |1+\tau Q(z)|=d>0.
$$
The inverse triangle inequality
$$
|1+\tau_n Q(z)|\ge|1+\tau Q(z)|-|\tau-\tau_n||Q(z)|
$$
together with the fact that $Q(z)$ is bounded on $\overline\cD$
implies that for some $n_0$ the relation
$$
\inf_{z\in\cD} |1+\tau_n  Q_{\tau}(z)|\ge d/2
$$
for all $n>n_0$ holds.
Consequently
$$
|Q_{\tau_n}(z)-Q_{\tau}(z)|
= \frac{|\tau-\tau_n|}{|1+\tau Q(z)|\ |1+\tau_nQ(z)|}\leq  4 d^{-2} |\tau-\tau_n|,
$$
which means that $Q_{\tau_n}(z)$ converges locally uniformly to $Q_{\tau}(z)$ on $\cD$.
By \cite[Theorem 1.4]{LaLuMa}, $Q_{\tau_n}(z)$ converges locally uniformly to $Q_{\tau}(z)$ on $\dC^+\setminus\{\beta(\tau)\}$.

Now consider the case $\tau=\infty$.
Since $\beta(\infty)=\alpha$, the function $Q(z)$  is bounded and bounded away from zero on $\overline\cD$. Consequently, the locally uniform convergence $Q_{\tau_n}(z)\to Q_{\infty}(z)$ follows from
$$
Q_{\tau_n}(z)-Q_{\infty}(z)=\frac{1}{1+\tau_n Q(z)}\left(Q(z)+\frac{1}{Q(z)}\right) \to 0,\quad n\to\infty,
$$
uniformly on $\cD$ and again \cite[Theorem 1.4]{LaLuMa}.

\end{proof}

\begin{theorem}\label{kurve}
The function $\tau \to \al(\tau)$ is continuous and the set
\[
\{\al(\tau):\,\tau \in \dR \cup \{\infty\}\}
\]
on the Riemann sphere is homeomorphic to a circle.
\end{theorem}

\begin{proof}
Let $(\tau_{n})$ be some sequence which converges to $\tau$.
By theorem \ref{cont}, the sequence $(Q_{\tau_n}(z))$
converges locally uniformly to
$Q_{\tau}(z)$ on $\dC^+\setminus\set{\beta(\tau)}$.
Now it follows from  \cite[Theorem 1.4]{LaLuMa}
that $\al(\tau_n)\to \al(\tau)$ if $n\to\infty$.
This shows that the function $\tau \to \al(\tau)$ is continuous.
Since it is also injective (\cite{SWW}) and the extended real line is compact on the Riemann sphere,
the inverse of $\tau \to \al(\tau)$ is continuous as well.
\end{proof}


Now the original topic about nonsymmetric extensions of Nevanlinna functions
is taken up again.

\begin{proposition}
Let $Q(z)$ be an $\mathbf{N}_1$ function.
Assume that  $\Omega$ is a simply connected domain with
$\Omega\cap\dR\neq\emptyset$  such that $\beta,\bar\beta\notin\Omega$,
and assume that $Q(z)$ extends to a  holomorphic function
$\widetilde Q(z)$ in $\Omega\cup\dC^+$. If the set
$$
A=\set{\al(\tau):\tau\in\dR\cup\set\infty}\cap\dR\cap\Omega
$$
has an accumulation point in $\Omega$,
then $\Omega\cap\dR$ is outside the support of $\sigma$.
\end{proposition}

\begin{proof}
Consider the function
$$
W(z)=\frac{\widetilde Q(z)+\overline{\widetilde Q(\bar z)}}2,
$$
which is holomorphic in $\Omega\cup\dC^+$ and real on $\Omega\cap\dR$.
Furthermore,
$$
\tilde Q(z)=W(z),\quad z\in A,
$$
since $Q(z)\in\dR$ for $z\in A$. Hence, $Q(z)=W(z)$ and if $z,\bar z\in\Omega$ then
$$
Q(\bar z)=W(\bar z)=W(z)=Q(z).
$$
Consequently, $\Omega\cap\dR$ is contained in the gap of $Q(z)$.
\end{proof}

\section{The behavior of $\al(\tau)$ meeting the real line}

Proposition \ref{zeros>0} below is a generalization  of \cite[Proposition 2.7]{SWW} for the case when $z_0\in\dR$ and the function $Q(z)$ extends holomorphically to a  holomorphic function $\tilde Q(z)$ in some simply connected neighborhood
$\Omega$ of $z_0$. Compared to \cite[Proposition 2.7]{SWW}, now it is not assumed  that the extension satisfies the symmetry principle, i.e. the point $z_0$ may lay in the support of $\sigma$. It will be frequently used that the $k$--th nontangential derivative at $z_0$ coincides with the derivative of the extension $\widetilde Q(z)$ at $z=z_0$.

\begin{proposition}\label{zeros>0}
Let $Q(z) \in \mathbf{N}_1$ and assume that
$Q(z)$ extends to a  holomorphic function  $\widetilde Q(z)$
in some simply connected neighborhood $\Omega$ of $z_0 \in \dR$.
If $z_0$ is the GZNT of $Q(z)$,
then  precisely one of the following cases occurs:
\begin{itemize}
\item[(1)] $\widetilde Q'(z_0) < 0$;
\item[(2)] $\widetilde Q'(z_0) = 0$ and $\tilde Q''(z_0) \neq 0$,
in which case $\IM \widetilde Q''(z_0)\geq0$;
\item[(3)] $\widetilde Q'(z_0) = 0$ and $\widetilde Q''(z_0) = 0$,
in which case $\widetilde Q'''(z_0) >0$.
\end{itemize}
\end{proposition}

\begin{proof}
According to  Theorem \ref{QExt}(ii) the extension $\widetilde Q(z)$ can be represented as follows:  
\begin{equation}\label{qwi}
\widetilde Q(z)=(z-z_0)^2\, R_{0}(z)\, \left(\frac{m_0}{z_0-z} + \widetilde M_1(z)\right),
\end{equation}
with $m_{0} \ge 0$ and
\[
R_{0}(z)=1/(z-\beta)(z-\bar\beta) \quad \mbox{or} \quad   R_{0}(z)\equiv 1,
\]
depending on the position of the GPNT.  The function  $\widetilde M_1(z)$
in \eqref{qwi}
is a Nevanlinna function in $\dC^{+}$ which is also
holomorphic in a neighborhood $\Omega$  of $z_{0}$ of the form
\[
\Omega=[z_{0}-\eps,z_{0}+\eps]+\ii[-\eps,\eps],
\]
where $\eps > 0$ is sufficiently small.

In order to list the possible cases, observe that it follows from \eqref{qwi} that
\[
\widetilde Q'(z_0)\neq 0 \quad \Leftrightarrow \quad m_0\neq 0,
\]
and in this situation $\widetilde Q'(z_0)=-m_0<0$.  This takes care of (1).

It remains to consider the situation $\widetilde Q'(z_0)=0$, in which case
it follows from \eqref{qwi} that
\begin{equation}\label{Q''}
\widetilde Q''(z_0)=2R_{0}(z_0)  {\widetilde M_1(z_0)}.
\end{equation}
Since $\widetilde M_1(z)$ is a Nevanlinna function in $\dC^{+}$
and it is continuous at $z_0$ it follows that $\IM \widetilde M_1(z_0) \ge 0$.
Furthermore, note that $R_{0}(z_0)>0$.  Therefore one has that
\begin{equation}\label{qwi+}
\widetilde Q'(z_0)=0 \quad \Rightarrow \quad \IM \widetilde Q''(z_0) \ge 0,
\end{equation}
which takes care of (2).

Finally, consider the case $\widetilde Q'(z_0)=\widetilde Q''(z_0)=0$.
Then by  \eqref{Q''} and the fact that $R_{0}(z_0)>0$
one has $\widetilde  M_1(z_0)=0$. Consequently,
\begin{equation}\label{qwi++}
\widetilde Q'''(z_0)=2R_{0}(z_0) \widetilde M_1'(z_0).
\end{equation}
Recall that by Theorem \ref{hExt} the function $\widetilde M_1(z)$
can be represented as
$$
\widetilde M_1(z)= \int_{z_0-\eps}^{z_0+\eps} \frac{\phi(t) dt}{t-z} + M_2(z),
$$
where $\phi(z)$ is a function holomorphic in $\Omega$,
and $M_2(z)$ is a Nevanlinna function with a gap $[z_0-\eps,z_0+\eps]$.
By the general theory of Nevanlinna functions \cite{donoghue} one has 
$$
\phi(z_0)=\frac 1\pi\ \IM \widetilde M_1(z_0)=0.
$$
Furthermore, $\phi'(z_0)=0$, since $\phi(t)$ is positive on $[z_0-\eps,z_0+\eps]$.
Hence, the function $\phi(t)/(z_0-t)^2$ is integrable on $[z_0-\eps,z_0+\eps]$. By the dominated convergence theorem it follows that
\begin{equation}\label{mic}
\widetilde M_1'(z_0)=\int_{z_0-\eps}^{z_0+\eps} \frac{\phi(t) dt}{(t-z_0)^2} + M_2'(z_0).
\end{equation}
It is clear that $M_2'(z_0) \ge 0$, since
$M_2(z)$ has a gap at $[z_0-\eps,z_0+\eps]$.
Hence \eqref{mic} shows that
$\widetilde M_1'(z_0) \ge 0$. In fact, at least one of
the terms in the right-hand side
of \eqref{mic} has
to be positive, otherwise $Q(z)\equiv 0$,
which is not an $\mathbf{N}_1$ function.
Hence it follows that $M_2'(z_0) > 0$ and,
by \eqref{qwi++}, $\widetilde Q'''(z_0)>0$.
This takes care of (3).
\end{proof}

In \cite[Theorem 4.1] {SWW} it is investigated what cases can occur if the curve $\{\al(\tau):\tau\in\dR\cup\set\infty\}$
meets the real line in a spectral gap of the function $M(z)$: Either it approaches the spectral gap perpendicular and then
continues through some subinterval of the spectral gap, or it approaches it with an angle of $\pi/3$,
hits the spectral gap in a single point and leaves it with an angle of $2\pi/3$. But it is also possible that the curve
$\{\al(\tau):\tau\in\dR\cup\set\infty\}$ meets the real line in a single point of $\supp\sigma$ and leaves it again, as
the examples in  Section 5 show. The following theorem provides a characterization of the possible cases in the
situation that $M(z)$ has a holomorphic extension through the corresponding point of intersection as in Proposition \ref{zeros>0}. 
The proof is similar to the proof of \cite[Theorem 4.1] {SWW} and again based on the generalized inverse function
theorem, see e.g. \cite[Theorem 9.4.3]{H}.

\begin{theorem} \label{mainth}
Let $Q(z) \in \mathbf{N}_1$ and assume that $Q(z)$
extends to a  holomorphic function $\widetilde Q(z)$ in some simply
connected neighborhood $\Omega$ of $z_0 \in \dR$. Furthermore
assume  that $\alpha(\tau_0)=z_{0}$ for $\tau_0\in\dR$.
Then precisely one of the following cases occur:
\begin{itemize}

\item[(1)]
$Q'(z_0)<0$.
Then there exists $\varepsilon > 0$ such that the function
$\alpha(\tau)$ is  holomorphic on $(-\varepsilon, \varepsilon)$ and
\[
 \lim_{\tau \uparrow 0} \arg (\alpha(\tau)-z_0) =0,
 \quad  \lim_{\tau \downarrow 0} \arg (\alpha(\tau)-z_0)  = \pi.
\]

\item[(2)]
$Q'(z_0)=0$ and $Q''(z_0) \neq 0$.
Then there exist $\varepsilon_1 > 0$ and $\varepsilon_2 > 0$ such that the function $\alpha(\tau)$ is holomorphic on each of the intervals $(-\varepsilon_1,0)$ and
$(0,\varepsilon_2)$. Moreover,
\begin{equation}\label{limits}
 \lim_{\tau \downarrow 0} \arg (\alpha(\tau)-z_0)  =\frac{2\pi - \theta_0}2,\quad   \lim_{\tau \uparrow 0} \arg (\alpha(\tau)-z_0) =  \frac{\pi - \theta_0}2   ,
\end{equation}
where $\theta_0=\arg Q''(z_0)$.

\item[(3)]
$Q'(z_0)=Q''(z_0)=0$ and $Q'''(z_0)\neq 0$.
Then there exist $\varepsilon_1 > 0$ and $\varepsilon_2 > 0$ such that
the function $\alpha(\tau)$ is holomorphic on each of the
intervals $(-\varepsilon_1,0)$ and $(0,\varepsilon_2)$. Moreover,
$$
\lim_{\tau \uparrow 0}\arg (\alpha(\tau)-z_0)=\frac\pi3,\quad \lim_{\tau \downarrow 0}\arg (\alpha(\tau)-z_0)= \frac{2\pi}3.
$$
\end{itemize}
\end{theorem}

\begin{proof}
Note that it is enough to consider the case $\tau=0$. Indeed,
if $Q(z)$ extends to some simply connected neigborhood
$\Omega$ of $z_0$ then  $Q_{\tau_0}(z)$ extends
to $\Omega\setminus\{\beta(\tau_0),\overline{\beta(\tau_0)}\}$.
Since $z_0=\al(\tau_0)\notin\{\be(\tau_0),\overline{\be(\tau_0)}\}$
one can choose a sufficiently small $\Omega$ for the function
$Q_{\tau_0}(z)$.  In this situation the cases (1) -- (3)
correspond precisely to the classification in
Proposition \ref{zeros>0}. Furthermore, without loosing generality,
it is assumed that $z_0=0$.

Case (1). According to the standard inverse function theorem,
there exists a function $\phi(w)$ satisfying $Q(\phi(w))=w$,
cf. \cite{SWW}.  Then define $\al(\tau)=\phi(\tau)$ for $\tau$ sufficiently small.
The power series of $\widetilde Q(z)$ at zero does need to have all its coefficients real,
as was the case in \cite{SWW}.

Case (2). $\widetilde Q(0)=\widetilde Q'(0)=0$, and $\IM\widetilde Q''(0) >0$.
According to the generalized inverse function theorem, see e.g. \cite[Theorem 9.4.3]{H}, the equation
\[
\widetilde Q(\phi^\pm(w))=w^2,
\]
has in some neighborhood of zero exactly two holomorphic solutions
$\phi^+(w)$ and $\phi^-(w)$.  The corresponding expansions
\[
\phi^\pm(w)=\phi_1^\pm w+\phi_2^\pm w^2+ \cdots ,
\]
satisfy $\phi_1^\pm=\pm (\widetilde Q''(0)/2)^{-1/2}$,
where the square root is chosen in such way that it transforms $\dC^{+}$
onto itself. Recall that  $\IM \widetilde Q''(0)\geq 0$ and, hence,
it follows that $\pm\IM\phi_1^\pm\geq 0$.

In the case  $\tau>0$ one has the identity
\begin{equation}\label{pmtau}
\widetilde Q(\phi^-(\tau^{1/2}))=\tau
\end{equation}
and $\arg\phi^-  =  \pi - \theta_0/2$.  Hence  $\phi^-(\tau^{1/2})$ is in
$\dC^{+} \cup \dR$ 
for small $\tau>0$. As a consequence one sees that
 $$
\alpha(\tau)= \phi^-(\tau^{1/2}),\quad 0<\tau<+\infty.
$$
The expansion of $\phi^-(\tau^{1/2})$ implies the   limit of $\arg (\alpha(\tau))$ as $\tau\downarrow0$:
\[
\begin{split}
\lim_{\tau\downarrow0}\tan(\arg(\alpha(\tau)))
&=\lim_{\tau\downarrow0}\frac{\IM\alpha(\tau)}{\RE\alpha(\tau)} \\
&= \frac{\IM\phi_1^{-}}{\RE\phi_1^{-}}
=\tan\arg\phi_1^- \\
& =\tan ( (2\pi - \theta_0)/2).
\end{split}
\]
Since the tangent function is injective on the interval $[0,\pi]$, the first part of \eqref{limits} follows.

Similarly, in the case $\tau<0$ one has the identity
\begin{equation}\label{pmtau2}
\widetilde Q(\phi^+(\ii|\tau^{1/2}|))=-|\tau|=\tau,
\end{equation}
and $\arg(\phi_1^+\ii)=(\pi-\theta_0)/2$. Hence,  $\phi^+((\ii|\tau^{1/2}|))$ is in
$\dC^{+}$
for small $\tau<0$. As a consequence one sees that
 $$
\alpha(\tau)= \phi^+(\tau^{1/2})\quad  -\infty<<\tau<0.
$$
 The expansion of $\phi^+(i|\tau|^{1/2})$ implies the  left limit of $\arg (\alpha(\tau))$ at zero:
\[
\lim_{\tau\uparrow0}\tan(\arg(\alpha(\tau)))=\lim_{\tau\uparrow0}\frac{\IM\alpha(\tau)}{\RE\alpha(\tau)}=
\tan(\arg(\ii\phi_1^+))
   =\tan ( (\pi - \theta_0)/2).
\]

Case (3) follows exactly along the same lines as in \cite{SWW}.
\end{proof}


\section{Classification of GZNT}\label{5}

Let $Q(z)$ belong to $\mathbf{N}_{1}$ and assume for simplicity
that its GPNT $\beta =\infty$. 
Then the integral representation of $Q(z)$ has the following form:
\begin{equation}\label{Qopform}
 Q(z)=(z-\al)(z-\bar \al)
 \left( a+bz+\int_{\dR}\left(\frac 1{t-z}-\frac{t}{t^2+1}\right)d\sigma(t)\right),
\end{equation}
with $\alpha \in \dC$, $a\in\dR$, $b\geq0$, and a measure $\sigma$ satisfying
$\int_{\dR} d\sigma(t)/(t^{2}+1)<\infty$.
If the GZNT $\alpha$ belongs to $\dR$, then
there is the following classification of $\alpha$ in terms of the integral representation
\eqref{Qopform}:
\begin{itemize}
\item[(A)] $\delta_\al :=\int _{\set{\al}} 1 d\sigma >0$;
\item[(B)] $\delta_\al=0$, $\int_{\dR} \frac{d\sigma(t)}{(t-\al)^2}=\infty$;
\item[(C)] $\delta_\al=0$, $\gamma_\al:=\lim_{z\widehat\to \al} \frac{Q(z)}{(z-\al)^2}\in\dR\setminus\{0\}$, $\int_{\dR} \frac{d\sigma(t)}{(t-\al)^2}<\infty$;
\item[(D)] $\delta_\al=\gamma_\al=0$, $\int_{\dR} \frac{d\sigma(t)}{(t-\al)^2}<\infty$, $\int_{\dR} \frac{d\sigma(t)}{(t-\al)^4}=\infty$;
\item[(E)] $\delta_\al=\gamma_\al=0$, $\int_{\dR} \frac{d\sigma(t)}{(t-\al)^4}<\infty$,
\end{itemize}
cf. \cite{DHS3}.
This classification has an interpretation in terms of a corresponding operator
model.
Let $H$ be selfadjoint operator in a Pontryagin space with inner product
$[\cdot,\cdot]$ 
with one negative square and let $\omega$ be a cyclic vector.
Associated with $H$ and $\omega$ is the function $Q(z)$ defined by\footnote{Note that $Q$ in the present paper plays the role of $Q_\infty$ in \cite{DHS3}. }
\begin{equation}\label{QHrel}
-1/Q(z) = [(H-z)^{-1}\omega,\omega],\quad z \in\rho(H);
\end{equation}
it can be shown that this function belongs to $\mathbf{ N}_1$.
Since $H$ is an operator, the function $Q(z)$ has a GPNT at infinity and therefore
it has the integral representation \eqref{Qopform}.
The GZNT $\alpha$ of $Q(z)$ is a real eigenvalue of nonpositive type of $H$.
The above classification of the real GZNT $\alpha$ corresponds to the
 equivalent classification (A)--(E) of the real eigenvalue $\alpha$ given in the Introduction.
cf. \cite[Theorem 5.1]{DHS3}.

The classification of the cases (1)--(3) in Theorem \ref{mainth} above will be compared with the above classification of the cases (A)--(E). In fact this comparison already
exists in  the  case of a spectral gap, i.e. in the case when the measure $\sigma$ is zero
around $z_0$ with a possible mass point at $z_0$; cf.  \cite{SWW}.
Here is the comparison in the general case; the possible pairings are listed
and examples of these pairings are indicated:

\begin{itemize}
\item[(A)] and (1):  an example with a nonsymmetric extension is given by
\[
Q(z)=z^2\left(\ii-\frac1z\right),
\]
see Example \ref{i}, while an example in the  case of a spectral gap is given by $Q(z)=-z$.
\item[(B)] and (2): an example with a nonsymmetric extension is given by
\[
Q(z)=z^2 \ \e^{\ii\theta_0}, \quad \theta_0\in[0,\pi],
\]
see Example \ref{ii}, while examples in the  case of a spectral gap do not exist.
\item[(C)] and (2):  an example with a nonsymmetric extension is given by
\[
Q(z)=z^2\left(  1+ \int_{-1}^1 \frac{t^2\  dt}{t-z}  \right),
\]
see Example \ref{iii}, while an example in the  case of a spectral gap is given by $Q(z)=z^{2}$.
\item[(D)] and (3): an example with a nonsymmetric extension is given by
\[
Q(z)=z^2\left(  \int_{-1}^1 \frac{t^2\  dt}{t-z}  \right),
\]
see Example \ref{iv}, while examples in the  case of a spectral gap do not exist.
\item[(E)] and (3): an example with a nonsymmetric extension is given by
\[
Q(z)=z^2\left(   \int_{-1}^1 \frac{t^4\  dt}{t-z}  \right),
\]
see Example \ref{v}, while an example in the  case of a spectral gap is given by $Q(z)=z^{3}$.
\end{itemize}
The rest of this section is devoted to the treatment of these and other examples.

%

%

\begin{example}\label{i}
 To illustrate Case (A) consider the $\mathbf{N}_1$ function
\[
 Q(z)={z^2}\, \left(\ii-\frac{1}{z} \right).
\]
The plot in Figure \ref{figi} shows all the points $z$ from the upper half--plane where $\IM Q(z)=0$. From Theorem \ref{mainth} Case (i) one can determine, that $\al(\tau)$ moves along the plotted curve from the right to the left hand side, i.e $\IM \Re (\al(\tau))$ is decreasing in $\tau$. Note that although $\al(\tau)$ approaches the origin horizontally, $\al(\tau)\notin\dR$ for $\tau\neq0$, in contrast to the the  case of a spectral gap described in \cite{SWW}. This behavior agrees with  \cite[Theorem 3.6]{SWW}.
\begin{figure}[htb]
\begin{center}
\includegraphics[width=200pt]{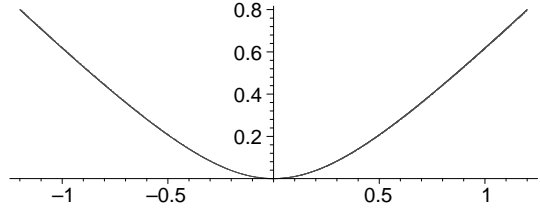}
\end{center}
\caption{Case (A), $Q(z)=z^2(\ii-1/z)$}\label{figi}
\end{figure}
An example, which is simpler to compute,
however not in the form \eqref{Qopform}, is
\[
 Q(z)=\frac{z^2}{(z-\ii)(z+\ii)}\, \left(\ii-\frac{1}{z} \right)=\frac{\ii z}{z-\ii}.
\]
Solving
$$
\frac{\ii z}{z-\ii}=\tau
$$
one gets
\[
 \alpha(\tau)=\frac{-\tau +\tau^2\ii}{\tau^2+1},
\]
and the same effect of approaching the origin tangentially from both sides is obtained.
\end{example}

\begin{example}\label{ii}
To illustrate Case (B) consider for $\theta_0\in[0,\pi]$ the function
 $$
Q(z)=z^2 \e^{\ii\theta_0}.
$$
Solving $Q(z)=\tau$ with $z\in\dC^+$ one obtains
$$
\al(\tau)= \begin{cases}
                  \sqrt{|\tau|} \cdot \e^{\ii (\pi - \theta_0)/2}, & \tau\leq 0, \\
                  \sqrt{\tau} \cdot \e^{\ii (2\pi - \theta_0)/2}, & \tau>0.
                  \end{cases}
                  $$
As another example of Case (B) consider
$$
Q(z)=z^2 \int_{-1}^{1} \frac{dt}{t-z}.
$$
Figure \ref{figii} contains the plot of points $z\in\dC^+$  satisfying $\IM Q(z)=0$.
\begin{figure}[htb]
\begin{center}
\includegraphics[width=200pt]{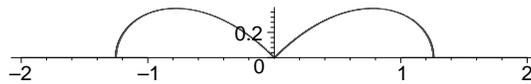}
\end{center}
\caption{Case (B), $Q(z)=z^2 \int_{-1}^{1} \frac{dt}{t-z}$.}\label{figii}
\end{figure}
Since $\alpha(\infty)=\infty$ one sees that for sufficiently small $\tau<0$  the point $\alpha(\tau)$ moves with increasing $\tau$ along the real line to the left until it reaches the point near $1.7$. There  it leaves the real line to the upper half--plane and continues along the plotted path until it reaches the real line again at approximately $-1.7$.
From that point it continues along the real line.
\end{example}

\begin{example}\label{iii}
To illustrate Case (C) consider the function
$$
Q(z)=z^2\left(  1+ \int_{-1}^1 \frac{t^2\  dt}{t-z}  \right).
$$
Figure \ref{figiii} contains the plot of points $z\in\dC^+$  satisfying $\IM Q(z)=0$.
One may observe that with $\tau\downarrow 0$ the point $\alpha(\tau)$ approches the real line approximately vertically and it leaves the origin approximately horizontally.
It is known from Theorem 3.6 of \cite{SWW} that the only point in a neighborhood of zero where $\al(\tau)\in\dR$ is the origin itself.
 \begin{figure}[htb]
\begin{center}
\includegraphics[width=200pt]{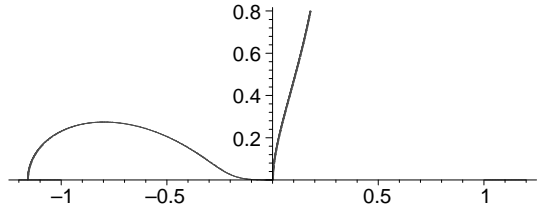}
\end{center}
\caption{Case (C), $Q(z)=z^2\left(  1+ \int_{-1}^1 \frac{t^2\  dt}{t-z}  \right)$}\label{figiii}
\end{figure}
\end{example}

\begin{example}\label{iv}
To illustrate Case (D) consider the function
 $$
 Q(z)=z^2\int_{-1}^1 \frac{t^2}{t-z}dt.
 $$
 Figure \ref{figiv} contains the plot of points $z\in\dC^+$  satisfying $\IM Q(z)=0$.
Note the essential difference between Figure \ref{figii} and Figure \ref{figiv}, in Figure \ref{figii} the angle between the left and right limit of $\al(\tau)$ at the origin is $\pi/2$, while in Figure \ref{figiv} it is $\pi/3$.
The movement of $\al(\tau)$ along the plotted line and the real line is the same as in Example \ref{ii}.
\begin{figure}[htb]
\begin{center}
\includegraphics[width=200pt]{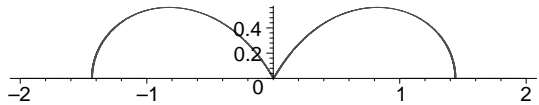}
\end{center}
\caption{Case (D), $Q(z)=z^2\left(  \int_{-1}^1 \frac{t^2\  dt}{t-z}  \right) $}\label{figiv}
\end{figure}
\end{example}
\newpage

\begin{example}\label{v}Finally, to illustrate Case (E) consider the function
$$
Q(z)=z^2\left(  \int_{-1}^1 \frac{t^4\  dt}{t-z}  \right).
$$
Figure \ref{figv} contains the plot of points $z\in\dC^+$  satisfying $\IM Q(z)=0$.
The angle between the left and right limit of $\al(\tau)$ at the origin is $\pi/3$ and the movement of $\al(\tau)$ along the plotted line and the real line is the same as in Examples \ref{ii} and \ref{iii}.

 \begin{figure}[htb]
\begin{center}
\includegraphics[width=200pt]{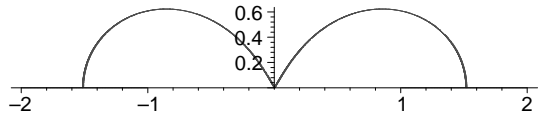}
\end{center}
\caption{Case (E), $Q(z)=z^2\left(  \int_{-1}^1 \frac{t^4\  dt}{t-z}  \right)$}\label{figv}
\end{figure}

\end{example}

\end{document}